\documentclass[12pt,reqno]{amsart}
\usepackage{amsmath}
\usepackage{amsfonts}
\usepackage{amssymb}
\usepackage{enumerate}
\usepackage{xcolor}

\newtheorem{theorem}{Theorem}

\newtheorem{corollary}[theorem]{Corollary}
\newtheorem{lemma}[theorem]{Lemma}

\allowdisplaybreaks

\begin{document}
\title[Explicit evaluations of the Hankel determinants]{Explicit evaluations of the Hankel determinants of a Thue--Morse-like sequence}
\date{May 4, 2014}
\author{Guo-Niu HAN}
\address{Institut de Recherche Math\'ematique Avanc\'ee\\
Universit\'e de Strasbourg et CNRS\\
7 rue Ren\'e-Descartes\\
 67084 Strasbourg\\
France}
\email{guoniu.han@unistra.fr}
\author{Wen WU}
\address{Department of Mathematics, Hubei University, 430062, Wuhan, P. R.
China}
\email[Corresponding author]{hust.wuwen@gmail.com}
\thanks{This research is supported by NSFC (Grant Nos. 11371156).}
\thanks{Wen Wu is the corresponding author.}
\begin{abstract}
We obtain the explicit evaluations of the Hankel determinants of the
formal power series
$\prod_{k\geq 0}(1+Jx^{3^{k}})$
where $J={(\sqrt{-3}-1)}/2$,
and prove that the sequence of Hankel determinants is an aperiodic automatic 
sequence taking value in $\{0, \pm 1, \pm J, \pm J^2\}$.
This research is essentially inspired by the works
about Hankel determinants of Thue--Morse-like sequences by
Allouche, Peyri\`ere, Wen and Wen (1998), Bacher (2006) and the first author (2013).
\end{abstract}
\maketitle

\section{Introduction} 
In 1998, Allouche, Peyri\`ere, Wen and Wen
proved that all the Hankel determinants of the Thue--Morse sequence
\begin{equation}
P_2(x)= \prod_{k\geq 0}(1-x^{2^{k}})
\end{equation}
are nonzero
by using
determinant manipulation \cite{APWW},
which consists of proving sixteen recurrent relations between determinants.
Recently, the first author derived
a short proof of APWW's result
by using the Jacobi continued fraction expansion of the undelying sequence \cite{Han}.
Moreover, he proved that
all the Hankel determinants of the following Thue--Morse-like sequence
\begin{equation}
P_3(x)= \prod_{k\geq 0}(1-x^{3^{k}})
\end{equation}
are nonzero.
Those results about Hankel determinants
have been shown to
have useful applications in Number Theory for
studing the irrationality exponents of automatic numbers (see \cite{Bug, GWW}).

Notice that the Hankel determinants of $P_2(x)$ and $P_3(x)$
do not have any closed-form expressions.
The trick to study such determiants
is using the modular arithmetic.
Surprisingly enough,
Bacher obtained the explicit evaluations
of the Hankel determinants of the following
Thue--Morse-like sequence
\begin{equation}
	P_2(x; I)=\prod_{k\geq 0} (1+Ix^{2^k}),
\end{equation}
where $I$ is the imaginary unit \cite{Ba06, Ba06a}.
Bacher's method is based on the {\it category Rec} introduced by himself.
Inspired by the above three results,
we derived the explicit evaluations of the
Hankel determinants of the following Thue--Morse-like sequence
\begin{equation}
	P_3(x;J)=\prod_{k\geq 0} (1+Jx^{3^k}),
\end{equation}
where $J=(I\sqrt3 -1)/2$.
Our proof is based on APWW's method by using
direct determinant manipulations.
\medskip

Consider the sequence $\boldsymbol{c}%
=(c_{0},c_{1},c_{2},\cdots )$ defined by the generating function
\begin{eqnarray}
&&	P_3(x;J)=\prod\limits_{k\geq 0}(1+Jx^{3^{k}})
	=c_{0}+c_{1}x+c_{2}x^{2}+\cdots \label{eq:defP3J}\\
 &=&1+{ J}x+{ J}{x}^{3}+{{ J}}^{2}{x}^{4}+{ J}{x
}^{9}+{{ J}}^{2}{x}^{10}
+{{ J}}^{2}{x}^{12}+{x}^{13}+{
J}{x}^{27}+\cdots \nonumber
\end{eqnarray}%
It is easy to show that
the sequence $\boldsymbol c$ can be characterized by the following recurrence relations
\begin{equation} \label{rec:c}
c_{0}=1,\ \ c_{3n}=c_{n},\ \ c_{3n+1}=Jc_{n},\ \ c_{3n+2}=0,\ \
\text{($n\geq 0$)}%
\end{equation}%
or equivalently by the morphism $\sigma $ over alphabet $\mathcal{A}%
=\{0,1,J,J^{2}\}$ where $\sigma $ is defined as follow%
\begin{equation*}
1\mapsto 1J0,\ \ J\mapsto JJ^{2}0,\ \ J^{2}\mapsto J^{2}10,\ \ 0\mapsto 000.
\end{equation*}%
Thus $\boldsymbol{c}$ is an \emph{automatic sequence}, i.e., it can be generated by a finite automaton  (see \cite{ALL}).

\bigskip

Recall that for each sequence of complex numbers
$\mathbf{u}=(u _{k})_{k=0,1,\ldots}$
the corresponding $(p,n)$-order
\emph{Hankel matrix}
$H_{n}^{p}(\mathbf u)$ is given by
\begin{equation}\label{def:H}
H_{n}^{p}(\mathbf u)=%
\begin{pmatrix}
u_{p} & u_{p+1} & \cdots & u_{p+n-1} \\
u_{p+1} & u_{p+2} & \cdots & u_{p+n} \\
\cdots & \cdots & \cdots & \cdots \\
u_{p+n-1} & u_{p+n} & \cdots & u_{p+2n-2}%
\end{pmatrix},%
\end{equation}%
where $n\geq 1$ and $p\geq 0$.
The {\it Hankel determinant} $|H_{n}^{p}(\mathbf u)|$ is simply  the determinant of
the Hankel matrix $H_{n}^{p}(\mathbf u)$. By convention, $|H_{0}^{p}(\mathbf u)|=1$.

\medskip
Our main result about the Hankel determinants of $\boldsymbol c$ is stated next.

\begin{theorem}
\label{thmH}
Let $H_{n}^{p}:=H_n^p(\boldsymbol c)$ be the $(p,n)$-order Hankel matrix of the Thue--Morse-like sequence
$\boldsymbol{c}$ defined in (\ref{eq:defP3J}).
Then,
the Hankel determinants $|H_n^0|$ and $|H_n^1|$
are characterized by the following recurrence relations
\begin{equation}\label{rec:H0}
\left\{
\begin{array}{cll}
|H_0^0|&=&1,   \\
	|H_{1}^{0}|  &=&  1,  \\
	|H_{3n}^{0}| &=& |H_{n}^{0}|\,, \\
	|H_{3n+1}^{0}| & = & |H_{n+1}^{0}|\,, \\
	|H_{3n+2}^{0}| &=&-J^2|H_{n+1}^{0}|\, \\
\end{array}%
\right.
\end{equation}
and
\begin{equation}\label{rec:H1}
\left\{
\begin{array}{cll}
|H_{0}^{1}|&=&1,   \\
|H_{3n}^{1}|&=&  |H_{n}^{1}|,\\
|H_{3n+1}^{1}|&=& J|H_{n}^{1}|,\\
|H_{3n+2}^{1}|&=& J |H_{n+1}^{1}|\\
\end{array}%
\right.
\end{equation}
for all  $n\geq 0$.
\end{theorem}

The first values of the Hankel determinants $|H_n^0|$ and $|H_n^1|$
are reproduced in the following table.
\begin{equation*}
	\begin{array}{*{14}c}
		n      &=& 0 & 1 & 2 & 3 & 4 & 5 & 6 & 7 & 8 & 9 & 10 & 11  \\
		|H_n^0|  &=&1&1&-J^2&1&-J^2&J&-J^2&1& -J^2&1&-J^2&J \\
		|H_n^1|  &=&1&J&J^2&J&J^2&1&J^ 2&1&J^2&J&J^2&1\\
\end{array}%
\end{equation*}
\medskip

Consider the sequence $\boldsymbol{s}%
=(s_{0},s_{1},s_{2},\cdots )$ defined by
$s_{n}=c_{n}+c_{n+1}$. We have
\begin{eqnarray}
	&&	S(x)=\frac{(1+x)P_3(x;J) -1}{x}
	=s_{0}+s_{1}x+s_{2}x^{2}+\cdots \qquad\qquad \label{eq:defS3J}\\
\nonumber &=&
-{{ J}}^{2}+{ J}x+{ J}{x}^{2}-{x}^{3}+{{ J}}^{
2}{x}^{4}
+{ J}{x}^{8}-{x}^{9}+{{ J}}^{2}{x}^{10}
+ \cdots \qquad \nonumber
\end{eqnarray}
and
\begin{equation}
s_{3n}=-J^{2}c_{n},~s_{3n+1}=Jc_{n},~s_{3n+2}=c_{n+1}.  \label{rec:s}
\end{equation}%
The Hankel matrices of the sequences $\mathbf{s}$ are denoted by $\Sigma
_{n}^{p}$. An easy observation shows that
\begin{equation}
\Sigma _{n}^{p}=H_{n}^{p}+H_{n}^{p+1}.
\end{equation}

\begin{theorem}
\label{thmSigma}
The Hankel determinants $|\Sigma_n^0|$ and $|\Sigma_n^1|$
are characterized by the following recurrent relations
\begin{equation}
\left\{
\begin{array}{cll}
|\Sigma_0^0|&=&1,   \\
	|\Sigma_{3n}^{0}| &=& |\Sigma_{n}^{0}|\,, \\
	|\Sigma_{3n+1}^{0}| & = & -J^2|\Sigma_{n}^{0}|\,, \\
	|\Sigma_{3n+2}^{0}| &=&-J^2|\Sigma_{n+1}^{0}|\, \\
\end{array}%
\right.  \label{Sreur0}
\end{equation}
and
\begin{equation}
\left\{
\begin{array}{cll}
|\Sigma_{0}^{1}|&=&1,   \\
|\Sigma_{3n}^{1}|&=&  |\Sigma_{n}^{1}|,\\
|\Sigma_{3n+1}^{1}|&=& J|\Sigma_{n}^{1}|,\\
|\Sigma_{3n+2}^{1}|&=& |\Sigma_{n}^{1}|\\
\end{array}%
\right.  \label{Sreur1}
\end{equation}
for all  $n\geq 0$.
\end{theorem}

The first values of the Hankel determinants $|\Sigma_n^0|$ and $|\Sigma_n^1|$
are reproduced in the following table.
\begin{equation*}
	\begin{array}{*{14}c}
		n      &=& 0 & 1 & 2 & 3 & 4 & 5 & 6 & 7 & 8 & 9 & 10 & 11  \\
		|\Sigma_n^0|  &=&1&-J^2&J&-J^2&J&-1&J&-1& J&-J^2&J&-1 \\
		|\Sigma_n^1|  &=&1&J&1&J&J^2&J&1&J&1&J&J^2&J\\
\end{array}%
\end{equation*}

\begin{theorem}
\label{thmAll}
For each $p,n\geq 0$,
the Hankel determinants $|H_n^p|$
and $|\Sigma_n^p|$ are equal to $0,\pm 1, \pm J, \pm J^2$.
\end{theorem}
\medskip

Let us make further comments about Hankel determinants and automatic sequences.
Hankel matrices and Hankel determinants
of a sequence
are
strong connected to the moment problem \cite{Sho}
and
to the Pad\'{e} approximation \cite{Bak, Bre}.
In \cite{Kam}, Kamae, Tamura and Wen studied the properties of Hankel
determinants for the Fibonacci word and give a quantitative relation between
the Hankel determinant and the Pad\'{e} pair. Later, Tamura \cite{Tam}
generalized the results for a class of special sequences. Allouche, Peyri%
\`{e}re, Wen and Wen studied the properties of Hankel determinants $|%
\mathcal{E}_n^p|$ of the Thue-Morse sequence in \cite{APWW}. They proved
that the Hankel determinants $|\mathcal{E}_n^p|$ modulo 2 recognized as a
two-dimensional sequence (or \emph{double sequence}) is 2-automatic.

Theorem \ref{thmH} implies that the first two columns of the two-dimensional
sequence $\{|H_{n}^{p}|\}_{n,p\geq 0}$, i.e., $\{|H_{n}^{0}|\}_{n\geq 0}$ and
$\{H_{n}^{1}\}_{n\geq 0}$, are 3-automatic sequences. They are obviously
\emph{aperiodic}, which are different than the
Hankel determinants of the Thue--Morse sequence and of the regular
paperfolding sequence studied in  \cite{APWW} and \cite {GWW, Han2} respectively.
\medskip

In Section 2 we establish a key lemma, namely, Lemma \ref{keylemma},
which consists a list of recurrent relations between
the determinants  $|H_n^p|$ and $|\Sigma_n^p|$.
Theorems \ref{thmH}-\ref{thmAll}
are simple consequences of Lemma \ref{keylemma}.
The recurrent relations them-self are proved in
Sections \ref{sec:proofkey} and \ref{sec:proofkey2}.

\section{The sudoku method} 
The proofs of Theorems \ref{thmH}-\ref{thmAll} are based
on the method developed
by Allouche, Peyri\`ere, Wen and Wen \cite{APWW}, that could be
called  {\it sudoku method}. The sudoku method consists some basic determinant
manipulations. Matrices are often split into $3\times 3=9$ small blocks.

For each matrix $M=(m_{i,j})_{i,j=1,2,\ldots, n}$ of size $n\times n$,
denote by $M^{t}$ the transpose of
$M$. Let $M^{(i)}$ be the $n\times (n-1)$-matrix obtained by deleting the $i$%
-th \emph{column} of $M$, and $M_{(i)}$ be the $(n-1)\times n$ -matrix
obtained by deleting the $i$-th \emph{row} of $M$.
The determinant of the matrix~$M$ is denoted by $|M|$.
Also, let $\mathbf{0}_{m,n}$ denote the $m\times n$ zero matrix.
\medskip

For each $n\geq 1$ let $P(n)$ be the $n\times n$-matrix definied by
\begin{equation}\label{def:Pn}
P(n)=(e_{1},e_{4},\cdots ,e_{3n_{1}-2},e_{2},e_{5},\cdots
,e_{3n_{2}-1},e_{3},e_{6},\cdots ,e_{3n_{3}}),
\end{equation}%
where $n_{1}=\lfloor\frac{n+2}{3}\rfloor,
n_{2}=\lfloor\frac{n+1}{3}\rfloor,
n_{3}=\lfloor\frac{n}{3}\rfloor$ and
$e_{j}$ is the $j$-th unit column vector of order $n$, i.e., the column
vector with $1$ as the $j$-th entry and zeros elsewhere.
For simplicity, we write $P$ instead of $P(n)$, when no confusion can occur.
Obviously, $|P(n)|=\pm 1$.
When consider $P(3n),P(3n+1),P(3n+2)$, the following diagram shows
the values of
$n_{1},n_{2}$ and $n_{3}$ in these cases:
\begin{equation*}
\begin{array}{c|ccc}
& n_{1} & n_{2} & n_{3} \\ \hline
3n & n & n & n \\
3n+1 & n+1 & n & n \\
3n+2 & n+1 & n+1 & n%
\end{array}%
\end{equation*}

\begin{lemma}\label{thm:PtMP}
Let $M=(m_{i,j})_{1\leq i,j\leq n}$ be an $n\times n$-matrix
and $P=P(n)$ be the matrix, of the same size as $M$, defined in (\ref{def:Pn}). Then
\begin{equation*}
P^{t}MP=\left(
\begin{matrix}
(m_{3i-2,3j-2})_{n_{1}\times n_{1}} & (m_{3i-2,3j-1})_{n_{1}\times n_{2}} &
(m_{3i-2,3j})_{n_{1}\times n_{3}} \\
(m_{3i-1,3j-2})_{n_{2}\times n_{1}} & (m_{3i-1,3j-1})_{n_{2}\times n_{2}} &
(m_{3i-1,3j})_{n_{2}\times n_{3}} \\
(m_{3i,3j-2})_{n_{3}\times n_{1}} & (m_{3i,3j-1})_{n_{3}\times n_{2}} &
(m_{3i,3j})_{n_{3}\times n_{3}}%
\end{matrix}%
\right) ,
\end{equation*}%
where $(m_{3i-2,3j-1})_{s\times t}$ means the matrix $(m_{3i-2,3j-1})_{1\leq
i\leq s,1\leq j\leq t}$.
\end{lemma}

Recall that for each sequence of complex numbers
$\mathbf{u}=(u _{k})_{k=0,1,\ldots}$
the corresponding $(p,n)$-order
\emph{Hankel matrix}
$H_{n}^{p}(\mathbf u)$ is defined by (\ref{def:H}).
Let $K_{n}^{p}=K_{n}^{p}(\mathbf u):=(u_{p+3(i+j-2)})_{1\leq i,j\leq n}$.
When applying Lemma \ref{thm:PtMP} for
$M=H_{3n}^{p}(\mathbf u)$, $H_{3n+1}^{p}(\mathbf u)$, $H_{3n+2}^{p}(\mathbf u)$, we get
\begin{eqnarray}
& & P^{t}H_{3n}^{p}(\mathbf u)P  \notag \\
&=&
\begin{pmatrix}
(u_{p+3(i+j-2)})_{n\times n} & (u_{p+3(i+j-2)+1})_{n\times n} &
(u_{p+3(i+j-2)+2})_{n\times n} \\
(u_{p+3(i+j-2)+1})_{n\times n} & (u_{p+3(i+j-2)+2})_{n\times n} &
(u_{p+3(i+j-2)+3})_{n\times n} \\
(u_{p+3(i+j-2)+2})_{n\times n} & (u_{p+3(i+j-2)+3})_{n\times n} &
(u_{p+3(i+j-2)+4})_{n\times n}%
\end{pmatrix}
\notag \\
&=&
\begin{pmatrix}
K_{n}^{p} & K_{n}^{p+1} & K_{n}^{p+2} \\
K_{n}^{p+1} & K_{n}^{p+2} & K_{n}^{p+3} \\
K_{n}^{p+2} & K_{n}^{p+3} & K_{n}^{p+4}%
\end{pmatrix}%
,  \label{h3n}
\end{eqnarray}
and
\begin{eqnarray}
P^{t}H_{3n+1}^{p}(\mathbf u)P &\!\!\!=\!\!\!\!&
\begin{pmatrix}
K_{n+1}^{p} & (K_{n+1}^{p+1})^{(n+1)} & (K_{n+1}^{p+2})^{(n+1)} \\
(K_{n+1}^{p+1})_{(n+1)} & K_{n}^{p+2} & K_{n}^{p+3} \\
(K_{n+1}^{p+2})_{(n+1)} & K_{n}^{p+3} & K_{n}^{p+4}%
\end{pmatrix}%
,  \label{h3n+1} \\
P^{t}H_{3n+2}^{p}(\mathbf u)P &\!\!\!=\!\!\!\!&
\begin{pmatrix}
K_{n+1}^{p} & K_{n+1}^{p+1} & (K_{n+1}^{p+2})^{(n+1)} \\
K_{n+1}^{p+1} & K_{n+1}^{p+2} & (K_{n+1}^{p+3})^{(n+1)} \\
(K_{n+1}^{p+2})_{(n+1)} & (K_{n+1}^{p+3})_{(n+1)} & K_{n}^{p+4}%
\end{pmatrix}%
.  \label{h3n+2}
\end{eqnarray}

\medskip

As shown in Sections \ref{sec:proofkey}-\ref{sec:proofkey2},
Formulae (\ref{h3n})-(\ref{h3n+2})
can be used to prove the following
recurrent relations between
the determinants  $|H_n^p|$ and $|\Sigma_n^p|$
when the sequence $\mathbf u$ is taken from the sequences
$\mathbf c$ and  $\mathbf s$ defined in (\ref{rec:c}) and (\ref{rec:s})
respectively.
Through these eighteen recurrent formulae, we can evaluate all the Hankel
determinants $|H_{n}^{p}|$ and $|\Sigma _{n}^{p}|$ $(n\geq 1,p\geq 0)$.
Our key lemma is stated next.

\begin{lemma}
\label{keylemma} For each $p\geq 0$ and $n\geq 1$ we have

\begin{enumerate}[$(L1)$\,]
\item $|H_{3n}^{3p}|=(-1)^{n}|H_{n}^{p}|\cdot |H_{n}^{p+1}|\cdot |\Sigma
_{n}^{p}|,$

\item $|H_{3n+1}^{3p}|=(-1)^{n}|H_{n}^{p+1}|\cdot |H_{n+1}^{p}|\cdot |\Sigma
_{n}^{p}|,$

\item $|H_{3n+2}^{3p}|=(-1)^{n+1}J^{2}\,|H_{n+1}^{p}|^{2}\cdot |\Sigma
_{n}^{p+1}|,$

\item $|H_{3n}^{3p+1}|=(-1)^{n}|H_{n}^{p+1}|^{2}\cdot |\Sigma _{n}^{p}|,$

\item $|H_{3n+1}^{3p+1}|=(-1)^{n}J\, |H_{n}^{p+1}|\cdot \left\vert
H_{n+1}^{p}\right\vert \cdot |\Sigma _{n}^{p+1}|,$

\item $|H_{3n+2}^{3p+1}|=(-1)^{n}J\, |H_{n+1}^{p+1}|\cdot \left\vert
H_{n+1}^{p}\right\vert \cdot |\Sigma _{n}^{p+1}|,$

\item $|H_{3n}^{3p+2}|=(-1)^{n}|H_{n}^{p+1}|^{2}\cdot |\Sigma _{n}^{p+1}|,$

\item $|H_{3n+1}^{3p+2}|=0,$

\item $|H_{3n+2}^{3p+2}|=(-1)^{n+1}|H_{n+1}^{p+1}|^{2}\cdot |\Sigma
_{n}^{p+1}|,$

\item $|\Sigma _{3n}^{3p}|=(-1)^{n}|\Sigma _{n}^{p}|^{2}\cdot |H_{n}^{p+1}|,$

\item $|\Sigma _{3n+1}^{3p}|=(-1)^{n+1}J^{2}\,|H_{n+1}^{p}|\cdot |\Sigma
_{n}^{p}|\cdot \left\vert \Sigma _{n}^{p+1}\right\vert ,$

\item $|\Sigma _{3n+2}^{3p}|=(-1)^{n+1}J^{2}\,|H_{n+1}^{p}|\cdot |\Sigma
_{n+1}^{p}|\cdot |\Sigma _{n}^{p+1}|,$

\item $|\Sigma _{3n}^{3p+1}|=(-1)^{n}|H_{n}^{p+1}|\cdot |\Sigma
_{n}^{p}|\cdot |\Sigma _{n}^{p+1}|,$

\item $|\Sigma _{3n+1}^{3p+1}|=(-1)^{n}J|H_{n+1}^{p}|\cdot |\Sigma
_{n}^{p+1}|^{2},$

\item $|\Sigma _{3n+2}^{3p+1}|=(-1)^{n+1}|H_{n+1}^{p+1}|\cdot |\Sigma
_{n+1}^{p}|\cdot |\Sigma _{n}^{p+1}|,$

\item $|\Sigma _{3n}^{3p+2}|=(-1)^{n}|\Sigma _{n}^{p+1}|^{2}\cdot
|H_{n}^{p+1}|,$

\item $|\Sigma _{3n+1}^{3p+2}|=(-1)^{n}|\Sigma _{n}^{p+1}|^{2}\cdot
|H_{n+1}^{p+1}|,$

\item $|\Sigma _{3n+2}^{3p+2}|=0.$
\end{enumerate}
\end{lemma}

\begin{corollary}\label{cor:AB}
Let $A_{n}^{p}=(-1)^{n}|H_{n}^{p+1}|\cdot |\Sigma_{n}^{p}|$
and $B_{n}^{p}=(-1)^{n}|H_{n+1}^{p}|\cdot |\Sigma _{n}^{p+1}|$. Then,
\begin{enumerate}[$(C1)$\,]
\item $A_{3n}^{3p}=(A_{n}^{p})^{3}\ ,$

\item $A_{3n+1}^{3p}=A_{n}^{p}(B_{n}^{p})^{2}\ ,$

\item $A_{3n+2}^{3p}=A_{n+1}^{p}(B_{n}^{p})^{2}\ ,$

\item $B_{3n}^{3p}=(A_{n}^{p})^{2}B_{n}^{p}\ ,$

\item $B_{3n+1}^{3p}=(B_{n}^{p})^{3}\ ,$

\item $B_{3n+2}^{3p}=(A_{n+1}^{p})^{2}B_{n}^{p}\ .$
\end{enumerate}
\end{corollary}

\begin{proof}
$(C1)$ From Equalities $(L4)$ and $(L10)$ stated in Lemma \ref{keylemma} we have
\begin{eqnarray*}
A_{3n}^{3p} &=&(-1)^{3n}|H_{3n}^{3p+1}|\cdot |\Sigma _{3n}^{3p}| \\
&=&(-1)^{3n}(-1)^{n}|H_{n}^{p+1}|^{2}|\Sigma _{n}^{p}|\cdot (-1)^{n}
|\Sigma _{n}^{p}|^{2} |H_{n}^{p+1}| \\
&=&(-1)^{3n}|H_{n}^{p+1}|^{3}|\Sigma _{n}^{p}|^{3} \\
&=&(A_{n}^{p})^{3}.
\end{eqnarray*}
Identity $(C2)$ (resp. $(C3)$, $(C4)$, $(C5)$, $(C6)$) are proved in the same manner by using
Equalities $(L5)$ and $(L11)$ (resp. $(L6)$ and $(L12)$,
$(L2)$ and $(L13)$, $(L3)$ and $(L14)$, $(L1)$ and $(L15)$ ) stated in Lemma \ref{keylemma}.
\end{proof}

\begin{proof}[Proof of Theorems \ref{thmH}-\ref{thmAll}]
Let $p_{n}=(-1)^{n}|H_{n}^{1}|\cdot |\Sigma _{n}^{0}|$ and $%
q_{n}=(-1)^{n}|H_{n+1}^{0}|\cdot |\Sigma _{n}^{1}|$. Then Corollory \ref{cor:AB} shows that

\begin{enumerate}
\item $p_{3n}=p_n^3$\,,

\item $p_{3n+1} = p_nq_n^2$\,,

\item $p_{3n+2} = p_{n+1} q_n^2$\,,

\item $q_{3n}=p_n^2q_n$\,,

\item $q_{3n+1}= q_n^3$\,,

\item $q_{3n+2}=p_{n+1}^2q_n$\,.

\end{enumerate}
Moreover, the first values are $p_0=p_1=p_2=q_0=q_1=q_2=1$.
By induction, we have $p_n=q_n=1$ for all
$n\geq 0$. In other words,
\begin{equation}\label{eq:HandS}
	|H_{n}^{1}|\cdot |\Sigma _{n}^{0}| =(-1)^n \text{\quad and \quad}
	|H_{n+1}^{0}|\cdot |\Sigma _{n}^{1}|=(-1)^n.
\end{equation}
The last three identities in (\ref{rec:H0})  are consequences of Equations
(L1), (L2), (L3) respectively.
Similarly,
the last three identities in (\ref{rec:H1})  are consequences of Equations
(L4), (L5), (L6) respectively.  Thus, Theorem \ref{thmH} is proved.
By Relation (\ref{eq:HandS}), Theorem \ref{thmH} implies Theorem \ref{thmSigma}.
Finally, Theorem \ref{thmAll} is a consequence of Lemma \ref{keylemma}, because
the set $\{0, \pm 1, \pm J, \pm J^2\}$ is closed under multiplication.
\end{proof}

\section{Proof of equalities  $(L1)$-$(L9)$}\label{sec:proofkey} 

Recall that the Thue--Morse-like sequence $\mathbf{c}=c_{0}c_{1}\cdots c_{n}\cdots \in
\mathcal{A}^{\mathbb{\infty }}$ is characterized by the recurrent
equations in (\ref{rec:c}),
and that
$H_{n}^{p}:=H_n^p(\boldsymbol c)$ is the $(p,n)$-order Hankel matrix of
$\boldsymbol{c}$.
Let $K_{n}^{p}:=K_{n}^{p}(\mathbf c):=(c_{p+3(i+j-2)})_{1\leq i,j\leq n}$.
By (\ref{rec:c}), we have for all $n\geq 1,p\geq 0,$
\begin{equation}
K_{n}^{3p}=H_{n}^{p},\quad
K_{n}^{3p+1}=JH_{n}^{p},\quad
K_{n}^{3p+2}=\boldsymbol{0}_{n,n}.  \label{eq:Kc}
\end{equation}%
Equalities $(L1)$-$(L8)$ are proved by combining (\ref{eq:Kc}) and (\ref{h3n}-\ref{h3n+2}) where the sequence $\mathbf u$ is specialized to  $\mathbf c$.
For simplicity, during the proof, we
will denote $n+1$ and $p+1$ by $\overline{n}$ and $\overline{p}$
respectively.
Also, let $I_{n,n}$ be the identity matrix of size $n\times n$.

\smallskip
\noindent
$(L1)$ Combine (\ref{eq:Kc}) and (\ref{h3n}), we have
\begin{flalign*}
	&	\left\vert H_{3n}^{3p}\right\vert =|P^{t}H_{3n}^{3p}P| \\
=&\left\vert
\begin{matrix}
H_{n}^{p} & JH_{n}^{p} & \boldsymbol{0}_{n,n} \\
JH_{n}^{p} & \boldsymbol{0}_{n,n} & H_{n}^{\overline{p}} \\
\boldsymbol{0}_{n,n} & H_{n}^{\overline{p}} & JH_{n}^{\overline{p}}%
\end{matrix}%
\right\vert \\
=&\left\vert \left(
\begin{matrix}
H_{n}^{p} & JH_{n}^{p} & \boldsymbol{0}_{n,n} \\
JH_{n}^{p} & \boldsymbol{0}_{n,n} & H_{n}^{\overline{p}} \\
\boldsymbol{0}_{n,n} & H_{n}^{\overline{p}} & JH_{n}^{\overline{p}}%
\end{matrix}%
\right) \left(
\begin{matrix}
I_{n,n} & -JI_{n,n} & \boldsymbol{0}_{n,n} \\
\boldsymbol{0}_{n,n} & I_{n,n} & \boldsymbol{0}_{n,n} \\
\boldsymbol{0}_{n,n} & -J^{2}I_{n,n} & I_{n,n}%
\end{matrix}%
\right) \right\vert \\
=&\left\vert
\begin{matrix}
H_{n}^{p} & \boldsymbol{0}_{n,n} & \boldsymbol{0}_{n,n} \\
JH_{n}^{p} & -J^{2}H_{n}^{p}-J^{2}H_{n}^{\overline{p}} & H_{n}^{\overline{p}}
\\
\boldsymbol{0}_{n,n} & \boldsymbol{0}_{n,n} & JH_{n}^{\overline{p}}%
\end{matrix}%
\right\vert \\
=&(-1)^{n}|H_{n}^{p}|\cdot |H_{n}^{p+1}|\cdot |\Sigma _{n}^{p}|. &
\end{flalign*}

\smallskip
\noindent
$(L2)$ Combine (\ref{eq:Kc}) and (\ref{h3n+1}), we have
\begin{flalign*}
	&	\left\vert H_{3n+1}^{3p}\right\vert  =|P^{t}H_{3n+1}^{3p}P| \\
=&\left\vert
\begin{matrix}
H_{\overline{n}}^{p} & \mathbf{(}JH_{\overline{n}}^{p}\mathbf{)}^{(\overline{%
n})} & \mathbf{0}_{\overline{n},n} \\
(JH_{\overline{n}}^{p})_{(\overline{n})} & \mathbf{0}_{n,n} & H_{n}^{%
\overline{p}} \\
\mathbf{0}_{n,\overline{n}} & H_{n}^{\overline{p}} & JH_{n}^{\overline{p}}%
\end{matrix}%
\right\vert  \\
=&\left\vert \left(
\begin{matrix}
H_{\overline{n}}^{p} & \mathbf{(}JH_{\overline{n}}^{p}\mathbf{)}^{(\overline{%
n})} & \mathbf{0}_{\overline{n},n} \\
(JH_{\overline{n}}^{p})_{(\overline{n})} & \mathbf{0}_{n,n} & H_{n}^{%
\overline{p}} \\
\mathbf{0}_{n,\overline{n}} & H_{n}^{\overline{p}} & JH_{n}^{\overline{p}}%
\end{matrix}%
\right) \left(
\begin{matrix}
I_{\overline{n},\overline{n}} & -JI_{\overline{n},n} & \mathbf{0}_{\overline{%
n},n} \\
\mathbf{0}_{n,\overline{n}} & I_{n,n} & 0 \\
\mathbf{0}_{n,\overline{n}} & -J^{2}I_{n,n} & I_{n,n}%
\end{matrix}%
\right) \right\vert  \\
=&\left\vert
\begin{matrix}
H_{\overline{n}}^{p} & \mathbf{0}_{\overline{n},n} & \mathbf{0}_{\overline{n}%
,n} \\
(JH_{\overline{n}}^{p})_{(\overline{n})} & -J^{2}(H_{n}^{p}+H_{n}^{\overline{%
p}}) & H_{n}^{\overline{p}} \\
\mathbf{0}_{n,\overline{n}} & \boldsymbol{0}_{n,n} & JH_{n}^{\overline{p}}%
\end{matrix}%
\right\vert  \\
=&(-1)^{n}|H_{n}^{p+1}|\cdot |H_{n+1}^{p}|\cdot |\Sigma _{n}^{p}|. &
\end{flalign*}

\smallskip
\noindent
$(L3)$ Combine (\ref{eq:Kc}) and (\ref{h3n+2}), we have
\begin{flalign*}
	&	\left\vert H_{3n+2}^{3p}\right\vert =|P^{t}H_{3n+2}^{3p}P| & \\
=&\left\vert
\begin{matrix}
H_{\overline{n}}^{p} & JH_{\overline{n}}^{p} & \boldsymbol{0}_{\overline{n}%
,n} \\
JH_{\overline{n}}^{p} & \boldsymbol{0}_{\overline{n},\overline{n}} & (H_{%
\overline{n}}^{\overline{p}})^{(\overline{n})} \\
\boldsymbol{0}_{n,\overline{n}} & (H_{\overline{n}}^{\overline{p}})_{(%
\overline{n})} & JH_{n}^{\overline{p}}%
\end{matrix}%
\right\vert \\
=&\left\vert \left(
\begin{matrix}
H_{\overline{n}}^{p} & JH_{\overline{n}}^{p} & \boldsymbol{0}_{\overline{n}%
,n} \\
JH_{\overline{n}}^{p} & \boldsymbol{0}_{\overline{n},\overline{n}} & (H_{%
\overline{n}}^{\overline{p}})^{(\overline{n})} \\
\boldsymbol{0}_{n,\overline{n}} & (H_{\overline{n}}^{\overline{p}})_{(%
\overline{n})} & JH_{n}^{\overline{p}}%
\end{matrix}%
\right) \left(
\begin{matrix}
I_{\overline{n},\overline{n}} & -JI_{\overline{n},\overline{n}} & -J^{2}(I_{%
\overline{n},\overline{n}})^{(1)} \\
\boldsymbol{0}_{\overline{n},\overline{n}} & I_{\overline{n},\overline{n}} &
J(I_{\overline{n},\overline{n}})^{(1)} \\
\boldsymbol{0}_{n,\overline{n}} & \boldsymbol{0}_{n,\overline{n}} & I_{n,n}%
\end{matrix}%
\right) \right\vert \\
=&\left\vert
\begin{matrix}
H_{\overline{n}}^{p} & \boldsymbol{0}_{\overline{n},\overline{n}} &
\boldsymbol{0}_{\overline{n},n} \\
JH_{\overline{n}}^{p} & -J^{2}H_{\overline{n}}^{p} & \boldsymbol{0}_{%
\overline{n},n} \\
\boldsymbol{0}_{n,\overline{n}} & (H_{\overline{n}}^{\overline{p}})_{(%
\overline{n})} & J(H_{n}^{\overline{p}}+H_{n}^{\overline{p}+1})%
\end{matrix}%
\right\vert \\
=&(-1)^{n+1}J^{2}\left\vert H_{n+1}^{p}\right\vert ^{2}\cdot \left\vert
\Sigma _{n}^{p+1}\right\vert .
\end{flalign*}

\smallskip
\noindent
$(L4)$ Combine (\ref{eq:Kc}) and (\ref{h3n}), we have
\begin{flalign*}
	&	\left\vert H_{3n}^{3p+1}\right\vert =|P^{t}H_{3n}^{3p+1}P| & \\
=&\left\vert
\begin{matrix}
JH_{n}^{p} & \mathbf{0}_{n,n} & H_{n}^{\overline{p}} \\
\mathbf{0}_{n,n} & H_{n}^{\overline{p}} & JH_{n}^{\overline{p}} \\
H_{n}^{\overline{p}} & JH_{n}^{\overline{p}} & \mathbf{0}_{n,n}%
\end{matrix}%
\right\vert \\
=&\left\vert \left(
\begin{matrix}
JH_{n}^{p} & \mathbf{0}_{n,n} & H_{n}^{\overline{p}} \\
\mathbf{0}_{n,n} & H_{n}^{\overline{p}} & JH_{n}^{\overline{p}} \\
H_{n}^{\overline{p}} & JH_{n}^{\overline{p}} & \mathbf{0}_{n,n}%
\end{matrix}%
\right) \left(
\begin{matrix}
I_{n,n} & -JI_{n,n} & J^{2}I_{n,n} \\
\mathbf{0}_{n,n} & I_{n,n} & -JI_{n,n} \\
\boldsymbol{0}_{n,n} & \boldsymbol{0}_{n,n} & I_{n,n}%
\end{matrix}%
\right) \right\vert \\
=&\left\vert
\begin{matrix}
JH_{n}^{p} & -J^{2}H_{n}^{p} & H_{n}^{p}+H_{n}^{\overline{p}} \\
\mathbf{0}_{n\times n} & H_{n}^{\overline{p}} & \boldsymbol{0}_{n,n} \\
H_{n}^{\overline{p}} & \boldsymbol{0}_{n,n} & \mathbf{0}_{n\times n}%
\end{matrix}%
\right\vert \\
=&(-1)^{n}\left\vert H_{n}^{p+1}\right\vert ^{2}\cdot \left\vert \Sigma
_{n}^{p}\right\vert .
\end{flalign*}

\smallskip
\noindent
$(L5)$ Combine (\ref{eq:Kc}) and (\ref{h3n+1}), we have
\begin{flalign*}
	&	\left\vert H_{3n+1}^{3p+1}\right\vert =|P^{t}H_{3n+1}^{3p+1}P| &\\
=&\left\vert
\begin{array}{ccc}
JH_{\overline{n}}^{p} & \boldsymbol{0}_{\overline{n},n} & (H_{\overline{n}}^{%
\overline{p}})^{(\overline{n})} \\
\boldsymbol{0}_{n,\overline{n}} & H_{n}^{\overline{p}} & JH_{n}^{\overline{p}%
} \\
(H_{\overline{n}}^{\overline{p}})_{(\overline{n})} & JH_{n}^{\overline{p}} &
\boldsymbol{0}_{n,n}%
\end{array}%
\right\vert \\
=&\left\vert \left(
\begin{array}{ccc}
JH_{\overline{n}}^{p} & \boldsymbol{0}_{\overline{n},n} & (H_{\overline{n}}^{%
\overline{p}})^{(\overline{n})} \\
\boldsymbol{0}_{n,\overline{n}} & H_{n}^{\overline{p}} & JH_{n}^{\overline{p}%
} \\
(H_{\overline{n}}^{\overline{p}})_{(\overline{n})} & JH_{n}^{\overline{p}} &
\boldsymbol{0}_{n,n}%
\end{array}%
\right) \left(
\begin{array}{ccc}
I_{n,n} & \boldsymbol{0}_{\overline{n},n} & -J^{2}(I_{\overline{n},\overline{%
n}})^{(1)} \\
\boldsymbol{0}_{n,\overline{n}} & I_{n,n} & -JI_{n,n} \\
\boldsymbol{0}_{n,\overline{n}} & \boldsymbol{0}_{n,n} & I_{n,n}%
\end{array}%
\right) \right\vert \\
=&%
\left|
\begin{array}{ccc}
JH_{\overline{n}}^{p} & \boldsymbol{0}_{\overline{n},n} & \boldsymbol{0}_{%
\overline{n},n} \\
\boldsymbol{0}_{n,\overline{n}} & H_{n}^{\overline{p}} & \boldsymbol{0}_{n,n}
\\
(H_{\overline{n}}^{\overline{p}})_{(\overline{n})} & JH_{n}^{\overline{p}} &
-J^{2}(H_{n}^{\overline{p}}+H_{n}^{\overline{p}+1})%
\end{array}
\right|
\\
=&(-1)^{n}J\cdot |H_{n}^{p+1}|\cdot \left\vert H_{n+1}^{p}\right\vert \cdot
|\Sigma _{n}^{p+1}|.
\end{flalign*}

\smallskip
\noindent
$(L6)$ Combine (\ref{eq:Kc}) and (\ref{h3n+2}), we have
\begin{flalign*}
	&\left\vert H_{3n+2}^{3p+1}\right\vert =\left\vert P^{t}H_{3n+2}^{3p+1}P\right\vert & \\
=&\left\vert
\begin{matrix}
JH_{\overline{n}}^{p} & \boldsymbol{0}_{\overline{n},\overline{n}} & (H_{%
\overline{n}}^{\overline{p}})^{(\overline{n})} \\
\boldsymbol{0}_{\overline{n},\overline{n}} & H_{\overline{n}}^{\overline{p}}
& (JH_{\overline{n}}^{\overline{p}})^{(\overline{n})} \\
(H_{\overline{n}}^{\overline{p}})_{(\overline{n})} & (JH_{\overline{n}}^{%
\overline{p}})_{(\overline{n})} & \boldsymbol{0}_{n,n}%
\end{matrix}%
\right\vert \\
=&\left\vert \left(
\begin{matrix}
JH_{\overline{n}}^{p} & \boldsymbol{0}_{\overline{n},\overline{n}} & (H_{%
\overline{n}}^{\overline{p}})^{(\overline{n})} \\
\boldsymbol{0}_{\overline{n},\overline{n}} & H_{\overline{n}}^{\overline{p}}
& (JH_{\overline{n}}^{\overline{p}})^{(\overline{n})} \\
(H_{\overline{n}}^{\overline{p}})_{(\overline{n})} & (JH_{\overline{n}}^{%
\overline{p}})_{(\overline{n})} & \boldsymbol{0}_{n,n}%
\end{matrix}%
\right) \right. \!\!
\left. \left(
\begin{array}{ccc}
I_{\overline{n},\overline{n}} & \boldsymbol{0}_{\overline{n},\overline{n}} &
-J^{2}(I_{\overline{n},\overline{n}})^{(1)} \\
\boldsymbol{0}_{\overline{n},\overline{n}} & I_{\overline{n},\overline{n}} &
-J(I_{\overline{n},\overline{n}})^{(\overline{n})} \\
\boldsymbol{0}_{n,\overline{n}} & \boldsymbol{0}_{n,\overline{n}} & I_{n,n}%
\end{array}%
\right) \right\vert \\
=&\left\vert
\begin{matrix}
JH_{\overline{n}}^{p} & \boldsymbol{0}_{\overline{n},\overline{n}} &
\boldsymbol{0}_{\overline{n},n} \\
\boldsymbol{0}_{\overline{n},\overline{n}} & H_{\overline{n}}^{p+1} &
\boldsymbol{0}_{\overline{n},n} \\
(H_{\overline{n}}^{\overline{p}})_{(\overline{n})} & (JH_{\overline{n}}^{%
\overline{p}})_{(\overline{n})} & -J^{2}(H_{n}^{\overline{p}+1}+H_{n}^{%
\overline{p}})%
\end{matrix}%
\right\vert \\
=&(-1)^{n}J\cdot |H_{n+1}^{p+1}|\cdot \left\vert H_{n+1}^{p}\right\vert
\cdot |\Sigma _{n}^{p+1}|.
\end{flalign*}

\smallskip
\noindent
$(L7)$ Combine (\ref{eq:Kc}) and (\ref{h3n}), we have
\begin{flalign*}
	&\left\vert H_{3n}^{3p+2}\right\vert =|P^{t}H_{3n}^{3p+2}P| &\\
=&\left\vert
\begin{matrix}
\boldsymbol{0}_{n,n} & H_{n}^{\overline{p}} & JH_{n}^{\overline{p}} \\
H_{n}^{\overline{p}} & JH_{n}^{\overline{p}} & \boldsymbol{0}_{n,n} \\
JH_{n}^{\overline{p}} & \boldsymbol{0}_{n,n} & H_{n}^{\overline{p}+1}%
\end{matrix}%
\right\vert \\
=&\left\vert \left(
\begin{matrix}
\boldsymbol{0}_{n,n} & H_{n}^{\overline{p}} & JH_{n}^{\overline{p}} \\
H_{n}^{\overline{p}} & JH_{n}^{\overline{p}} & \boldsymbol{0}_{n,n} \\
JH_{n}^{\overline{p}} & \boldsymbol{0}_{n,n} & H_{n}^{\overline{p}+1}%
\end{matrix}%
\right) \left(
\begin{array}{ccc}
I_{n,n} & \boldsymbol{0}_{n,n} & \boldsymbol{0}_{n,n} \\
-J^{2}I_{n,n} & I_{n,n} & \boldsymbol{0}_{n,n} \\
JI_{n,n} & -J^{2}I_{n,n} & I_{n,n}%
\end{array}%
\right) \right\vert \\
=&\left\vert
\begin{matrix}
\boldsymbol{0}_{n,n} & \boldsymbol{0}_{n,n} & JH_{n}^{\overline{p}} \\
\boldsymbol{0}_{n,n} & JH_{n}^{\overline{p}} & \boldsymbol{0}_{n,n} \\
J(H_{n}^{\overline{p}}+H_{n}^{\overline{p}+1}) & -J^{2}H_{n}^{\overline{p}+1}
& H_{n}^{\overline{p}+1}%
\end{matrix}%
\right\vert \\
=&(-1)^{n}|H_{n}^{p+1}|^{2}\cdot |\Sigma _{n}^{p+1}|.
\end{flalign*}

\smallskip
\noindent
$(L8)$ Combine (\ref{eq:Kc}) and (\ref{h3n+1}), we have
\begin{flalign*}
	&\left\vert H_{3n+1}^{3p+2}\right\vert =|P^{t}H_{3n+1}^{3p+2}P| &\\
=&\left\vert
\begin{array}{ccc}
\boldsymbol{0}_{\overline{n},\overline{n}} & (H_{\overline{n}}^{\overline{p}%
})^{(\overline{n})} & (JH_{\overline{n}}^{\overline{p}})^{(\overline{n})} \\
(H_{\overline{n}}^{\overline{p}})_{(\overline{n})} & JH_{n}^{\overline{p}} &
\mathbf{0}_{n\times n} \\
(JH_{\overline{n}}^{\overline{p}})_{(\overline{n})} & \mathbf{0}_{n\times n}
& H_{n}^{\overline{p}+1}%
\end{array}%
\right\vert \\
=&\left\vert
\begin{array}{ccc}
\boldsymbol{0}_{\overline{n},\overline{n}} & (H_{\overline{n}}^{\overline{p}%
})^{(\overline{n})} & (JH_{\overline{n}}^{\overline{p}})^{(\overline{n})} \\
(H_{\overline{n}}^{\overline{p}})_{(\overline{n})} & JH_{n}^{\overline{p}} &
\mathbf{0}_{n\times n} \\
\boldsymbol{0}_{n,\overline{n}} & -J^{2}H_{n}^{\overline{p}} & H_{n}^{%
\overline{p}+1}%
\end{array}%
\right\vert \\
=&0.
\end{flalign*}

\smallskip
\noindent
    $(L9)$ Combine (\ref{eq:Kc}) and (\ref{h3n+2}), we have
    \begin{flalign*}
    	&	\left\vert \Sigma _{3n+2}^{3p+2}\right\vert =|P^{t}\Sigma _{3n+2}^{3p+2}P|&
    \\
    =&\left\vert
    \begin{array}{ccc}
    \boldsymbol{0}_{\overline{n},\overline{n}} & H_{\overline{n}}^{\overline{p}}
    & \mathbf{(}JH_{\overline{n}}^{\overline{p}}\mathbf{)}^{(\overline{n})} \\
    H_{\overline{n}}^{\overline{p}} & JH_{\overline{n}}^{\overline{p}} &
    \boldsymbol{0}_{\overline{n},n} \\
    \mathbf{(}JH_{\overline{n}}^{\overline{p}}\mathbf{)}_{(\overline{n})} &
    \boldsymbol{0}_{n,\overline{n}} & H_{n}^{p+2}%
    \end{array}%
    \right\vert \\
    =&\left\vert \left(
    \begin{array}{ccc}
    \boldsymbol{0}_{\overline{n},\overline{n}} & H_{\overline{n}}^{\overline{p}}
    & \mathbf{(}JH_{\overline{n}}^{\overline{p}}\mathbf{)}^{(\overline{n})} \\
    H_{\overline{n}}^{\overline{p}} & JH_{\overline{n}}^{\overline{p}} &
    \boldsymbol{0}_{\overline{n},n} \\
    \mathbf{(}JH_{\overline{n}}^{\overline{p}}\mathbf{)}_{(\overline{n})} &
    \boldsymbol{0}_{n,\overline{n}} & H_{n}^{p+2}%
    \end{array}%
    \right) \right.
     \left. \left(
    \begin{array}{ccc}
    I_{\overline{n},\overline{n}} & -JI_{\overline{n},\overline{n}} & J^{2}I_{%
    \overline{n},n} \\
    \boldsymbol{0}_{\overline{n},\overline{n}} & I_{\overline{n},\overline{n}} &
    -JI_{\overline{n},n} \\
    \boldsymbol{0}_{n,\overline{n}} & \boldsymbol{0}_{n,\overline{n}} & I_{n,n}%
    \end{array}%
    \right) \right\vert \\
    =&\left\vert
    \begin{array}{ccc}
    \boldsymbol{0}_{\overline{n},\overline{n}} & H_{\overline{n}}^{\overline{p}}
    & \mathbf{0}_{\overline{n},n} \\
    H_{\overline{n}}^{\overline{p}} & \boldsymbol{0}_{\overline{n},\overline{n}}
    & \boldsymbol{0}_{\overline{n},n} \\
    \mathbf{(}JH_{\overline{n}}^{\overline{p}}\mathbf{)}_{(\overline{n})} &
    \mathbf{-}J^{2}\mathbf{(}H_{\overline{n}}^{\overline{p}}\mathbf{)}_{(%
    \overline{n})} & H_{n}^{\overline{p}}+H_{n}^{\overline{p}+1}%
    \end{array}%
    \right\vert \\
    =&(-1)^{n+1}|H_{n+1}^{p+1}|^{2}\cdot |\Sigma _{n}^{p+1}|.
    \end{flalign*}

\section{Proof of equalities  $(L10)$-$(L18)$}\label{sec:proofkey2} 

Recall that the  sequence $\mathbf{s}=s_{0}s_{1}\cdots s_{n}\cdots \in
\mathcal{A}^{\mathbb{\infty }}$ is characterized by the recurrent
equations in (\ref{rec:s}),
and that
$\Sigma_{n}^{p}:=\Sigma_n^p(\boldsymbol s)$ (resp. $H_n^p$)
is the $(p,n)$-order Hankel matrix of   the sequence
$\boldsymbol{s}$ (resp. $\boldsymbol c$).
Let $K_{n}^{p}:=K_{n}^{p}(\mathbf s):=(s_{p+3(i+j-2)})_{1\leq i,j\leq n}$.
By (\ref{rec:s}), we have for all $n\geq 1,p\geq 0,$
\begin{equation}
K_{n}^{3p}=-J^{2}H_{n}^{p},\quad
K_{n}^{3p+1}=JH_{n}^{p},\quad
K_{n}^{3p+2}=H_{n}^{p+1}.
\label{eq:Ks}
\end{equation}
Equalities $(L10)$-$(L18)$ are proved by combining (\ref{eq:Ks}) and (\ref{h3n}-\ref{h3n+2}) where the sequence $\mathbf u$ is specialized to  $\mathbf s$.

\smallskip
\noindent
$(L10)$ Combine (\ref{eq:Ks}) and (\ref{h3n}), we have
\begin{flalign*}
	&\left\vert \Sigma _{3n}^{3p}\right\vert =|P^{t}\Sigma _{3n}^{3p}P| & \\
=&\left\vert
\begin{array}{ccc}
-J^{2}H_{n}^{p} & JH_{n}^{p} & H_{n}^{\overline{p}} \\
JH_{n}^{p} & H_{n}^{\overline{p}} & -J^{2}H_{n}^{\overline{p}} \\
H_{n}^{\overline{p}} & -J^{2}H_{n}^{\overline{p}} & JH_{n}^{\overline{p}}%
\end{array}%
\right\vert \\
=&\left\vert \left(
\begin{array}{ccc}
-J^{2}H_{n}^{p} & JH_{n}^{p} & H_{n}^{\overline{p}} \\
JH_{n}^{p} & H_{n}^{\overline{p}} & -J^{2}H_{n}^{\overline{p}} \\
H_{n}^{\overline{p}} & -J^{2}H_{n}^{\overline{p}} & JH_{n}^{\overline{p}}%
\end{array}%
\right) \left(
\begin{array}{ccc}
I_{n,n} & J^{2}I_{n,n} & \boldsymbol{0}_{n,n} \\
\boldsymbol{0}_{n,n} & I_{n,n} & J^{2}I_{n,n} \\
\boldsymbol{0}_{n,n} & \boldsymbol{0}_{n,n} & I_{n,n}%
\end{array}%
\right) \right\vert \\
=&\left\vert
\begin{array}{ccc}
-J^{2}H_{n}^{p} & \boldsymbol{0}_{n,n} & H_{n}^{p}+H_{n}^{\overline{p}} \\
JH_{n}^{p} & H_{n}^{p}+H_{n}^{\overline{p}} & \boldsymbol{0}_{n,n} \\
H_{n}^{\overline{p}} & \boldsymbol{0}_{n,n} & \boldsymbol{0}_{n,n}%
\end{array}%
\right\vert \\
=&(-1)^{n}|\Sigma _{n}^{p}|^{2}\cdot |H_{n}^{p+1}|.
\end{flalign*}

\smallskip
\noindent
$(L11)$ Combine (\ref{eq:Ks}) and (\ref{h3n+1}), we have
\begin{flalign*}
	&\left\vert \Sigma _{3n+1}^{3p}\right\vert =\left\vert P^{t}\Sigma
	_{3n+1}^{3p}P\right\vert& \\
=&\left\vert
\begin{array}{ccc}
-J^{2}H_{\overline{n}}^{p} & (JH_{\overline{n}}^{p})^{(\overline{n})} & (H_{%
\overline{n}}^{\overline{p}})^{(\overline{n})} \\
(JH_{\overline{n}}^{p})_{(\overline{n})} & H_{n}^{\overline{p}} &
-J^{2}H_{n}^{\overline{p}} \\
(H_{\overline{n}}^{\overline{p}})_{(\overline{n})} & -J^{2}H_{n}^{\overline{p%
}} & JH_{n}^{\overline{p}}%
\end{array}%
\right\vert \\
=&\left\vert \left(
\begin{array}{ccc}
-J^{2}H_{\overline{n}}^{p} & (JH_{\overline{n}}^{p})^{(\overline{n})} & (H_{%
\overline{n}}^{\overline{p}})^{(\overline{n})} \\
(JH_{\overline{n}}^{p})_{(\overline{n})} & H_{n}^{\overline{p}} &
-J^{2}H_{n}^{\overline{p}} \\
(H_{\overline{n}}^{\overline{p}})_{(\overline{n})} & -J^{2}H_{n}^{\overline{p%
}} & JH_{n}^{\overline{p}}%
\end{array}%
\right) \right. \\
&\qquad \times \left. \left(
\begin{array}{ccc}
I_{\overline{n},\overline{n}} & J^{2}(I_{\overline{n},\overline{n}})^{(%
\overline{n})} & J(I_{\overline{n},\overline{n}})^{(1)} \\
\boldsymbol{0}_{n,\overline{n}} & I_{n,n} & \boldsymbol{0}_{n,n} \\
\boldsymbol{0}_{n,\overline{n}} & \boldsymbol{0}_{n,n} & I_{n,n}%
\end{array}%
\right) \right\vert \\
=&\left\vert
\begin{array}{ccc}
-J^{2}H_{\overline{n}}^{p} & \boldsymbol{0}_{\overline{n},n} & \boldsymbol{0}%
_{\overline{n},n} \\
(JH_{\overline{n}}^{p})_{(\overline{n})} & H_{n}^{p}+H_{n}^{\overline{p}} &
\boldsymbol{0}_{n,n} \\
(H_{\overline{n}}^{\overline{p}})_{(\overline{n})} & \boldsymbol{0}_{n,n} &
J(H_{n}^{\overline{p}}+H_{n}^{\overline{p}+1})%
\end{array}%
\right\vert \\
=&(-1)^{n+1}J^{2}|H_{n+1}^{p}|\cdot |\Sigma _{n}^{p}|\cdot \left\vert
\Sigma _{n}^{p+1}\right\vert .
\end{flalign*}

\smallskip
\noindent
$(L12)$ Combine (\ref{eq:Ks}) and (\ref{h3n+2}), we have
\begin{flalign*}
	&\left\vert \Sigma _{3n+2}^{3p}\right\vert =\left\vert P^{t}\Sigma
	_{3n+2}^{3p}P\right\vert & \\
=&\left\vert
\begin{array}{ccc}
-J^{2}H_{\overline{n}}^{p} & JH_{\overline{n}}^{p} & \left( H_{\overline{n}%
}^{\overline{p}}\right) ^{(\overline{n})} \\
JH_{\overline{n}}^{p} & H_{\overline{n}}^{\overline{p}} & -J^{2}\left( H_{%
\overline{n}}^{\overline{p}}\right) ^{(\overline{n})} \\
\left( H_{\overline{n}}^{\overline{p}}\right) _{(\overline{n})} &
-J^{2}\left( H_{\overline{n}}^{\overline{p}}\right) _{(\overline{n})} &
JH_{n}^{\overline{p}}%
\end{array}%
\right\vert \\
=&\left\vert \left(
\begin{array}{ccc}
-J^{2}H_{\overline{n}}^{p} & JH_{\overline{n}}^{p} & \left( H_{\overline{n}%
}^{\overline{p}}\right) ^{(\overline{n})} \\
JH_{\overline{n}}^{p} & H_{\overline{n}}^{\overline{p}} & -J^{2}\left( H_{%
\overline{n}}^{\overline{p}}\right) ^{(\overline{n})} \\
\left( H_{\overline{n}}^{\overline{p}}\right) _{(\overline{n})} &
-J^{2}\left( H_{\overline{n}}^{\overline{p}}\right) _{(\overline{n})} &
JH_{n}^{\overline{p}}%
\end{array}%
\right) \right. \\
&\qquad \times \left. \left(
\begin{array}{ccc}
I_{\overline{n},\overline{n}} & J^{2}I_{\overline{n},\overline{n}} & J(I_{%
\overline{n},\overline{n}})^{(1)} \\
\boldsymbol{0}_{\overline{n},\overline{n}} & I_{\overline{n},\overline{n}} &
\boldsymbol{0}_{\overline{n},n} \\
\boldsymbol{0}_{n,\overline{n}} & \boldsymbol{0}_{n,n+1} & I_{n,n}%
\end{array}%
\right) \right\vert \\
=&\left\vert
\begin{array}{ccc}
-J^{2}H_{\overline{n}}^{p} & \boldsymbol{0}_{\overline{n},\overline{n}} &
\boldsymbol{0}_{\overline{n},n} \\
JH_{\overline{n}}^{p} & H_{\overline{n}}^{\overline{p}} & \boldsymbol{0}_{%
\overline{n},n} \\
\left( H_{\overline{n}}^{\overline{p}}\right) _{(\overline{n})} &
\boldsymbol{0}_{n,\overline{n}} & J(H_{n}^{\overline{p}}+H_{n}^{\overline{p}%
+1})%
\end{array}%
\right\vert \\
=&(-1)^{n+1}J^{2}|H_{n+1}^{p}|\cdot |\Sigma _{n+1}^{p}|\cdot |\Sigma
_{n}^{p+1}|.
\end{flalign*}

\smallskip
\noindent
$(L13)$ Combine (\ref{eq:Ks}) and (\ref{h3n}), we have
\begin{flalign*}
	&	\left\vert \Sigma _{3n}^{3p+1}\right\vert =\left\vert P^{t}\Sigma
	_{3n}^{3p+1}P\right\vert & \\
 =&\left\vert
\begin{matrix}
JH_{n}^{p} & H_{n}^{\overline{p}} & \mathbf{-}J^{2}H_{n}^{\overline{p}} \\
H_{n}^{\overline{p}} & \mathbf{-}J^{2}H_{n}^{\overline{p}} & JH_{n}^{%
\overline{p}} \\
\mathbf{-}J^{2}H_{n}^{\overline{p}} & JH_{n}^{\overline{p}} & H_{n}^{%
\overline{p}+1}%
\end{matrix}%
\right\vert \\
=&\left\vert \left(
\begin{matrix}
JH_{n}^{p} & H_{n}^{\overline{p}} & \mathbf{-}J^{2}H_{n}^{\overline{p}} \\
H_{n}^{\overline{p}} & \mathbf{-}J^{2}H_{n}^{\overline{p}} & JH_{n}^{%
\overline{p}} \\
\mathbf{-}J^{2}H_{n}^{\overline{p}} & JH_{n}^{\overline{p}} & H_{n}^{%
\overline{p}+1}%
\end{matrix}%
\right) \left(
\begin{array}{ccc}
I_{n,n} & \mathbf{0}_{n\times n} & \mathbf{0}_{n\times n} \\
JI_{n,n} & I_{n,n} & J^{2}I_{n,n} \\
\mathbf{0}_{n\times n} & \mathbf{0}_{n\times n} & I_{n,n}%
\end{array}%
\right) \right\vert \\
=&\left\vert
\begin{matrix}
J(H_{n}^{p}+H_{n}^{\overline{p}}) & H_{n}^{\overline{p}} & \mathbf{0}%
_{n\times n} \\
\mathbf{0}_{n\times n} & \mathbf{-}J^{2}H_{n}^{\overline{p}} & \mathbf{0}%
_{n\times n} \\
\mathbf{0}_{n\times n} & JH_{n}^{\overline{p}} & H_{n}^{\overline{p}}+H_{n}^{%
\overline{p}+1}%
\end{matrix}%
\right\vert \\
=&(-1)^{n}|H_{n}^{p+1}|\cdot |\Sigma _{n}^{p}|\cdot |\Sigma _{n}^{p+1}|.
\end{flalign*}

\smallskip
\noindent
$(L14)$ Combine (\ref{eq:Ks}) and (\ref{h3n+1}), we have
\begin{flalign*}
	&\left\vert \Sigma _{3n+1}^{3p+1}\right\vert =\left\vert P^{t}\Sigma
	_{3n+1}^{3p+1}P\right\vert & \\
=&\left\vert
\begin{matrix}
JH_{\overline{n}}^{p} & (H_{\overline{n}}^{\overline{p}})^{(\overline{n})} &
\mathbf{-}J^{2}(H_{\overline{n}}^{\overline{p}})^{(\overline{n})} \\
(H_{\overline{n}}^{\overline{p}})_{(\overline{n})} & \mathbf{-}J^{2}H_{n}^{%
\overline{p}} & JH_{n}^{\overline{p}} \\
\mathbf{-}J^{2}(H_{\overline{n}}^{\overline{p}})_{(\overline{n})} & JH_{n}^{%
\overline{p}} & H_{n}^{p+2}%
\end{matrix}%
\right\vert \\
=&\left\vert \left(
\begin{matrix}
JH_{\overline{n}}^{p} & (H_{\overline{n}}^{\overline{p}})^{(\overline{n})} &
\mathbf{-}J^{2}(H_{\overline{n}}^{\overline{p}})^{(\overline{n})} \\
(H_{\overline{n}}^{\overline{p}})_{(\overline{n})} & \mathbf{-}J^{2}H_{n}^{%
\overline{p}} & JH_{n}^{\overline{p}} \\
\mathbf{-}J^{2}(H_{\overline{n}}^{\overline{p}})_{(\overline{n})} & JH_{n}^{%
\overline{p}} & H_{n}^{p+2}%
\end{matrix}%
\right) \right. \\
&\qquad \times \left. \left(
\begin{array}{ccc}
I_{n,n} & -J^{2}(I_{\overline{n},\overline{n}})^{(1)} & J(I_{\overline{n},%
\overline{n}})^{(1)} \\
\boldsymbol{0}_{n,\overline{n}} & I_{n,n} & \boldsymbol{0}_{n,n} \\
\boldsymbol{0}_{n,\overline{n}} & \boldsymbol{0}_{n,n} & I_{n,n}%
\end{array}%
\right) \right\vert \\
=&\left\vert
\begin{matrix}
JH_{\overline{n}}^{p} & \boldsymbol{0}_{n+1,n} & \boldsymbol{0}_{\overline{n}%
,n} \\
(H_{\overline{n}}^{\overline{p}})_{(\overline{n})} & \mathbf{-}J^{2}(H_{n}^{%
\overline{p}}+H_{n}^{\overline{p}+1}) & J(H_{n}^{\overline{p}}+H_{n}^{%
\overline{p}+1}) \\
\mathbf{-}J^{2}(H_{\overline{n}}^{\overline{p}})_{(\overline{n})} & J(H_{n}^{%
\overline{p}}+H_{n}^{\overline{p}+1}) & \boldsymbol{0}_{n,n}%
\end{matrix}%
\right\vert \\
=&(-1)^{n}J|H_{n+1}^{p}|\cdot |\Sigma _{n}^{p+1}|^{2}.
\end{flalign*}

\smallskip
\noindent
$(L15)$ Combine (\ref{eq:Ks}) and (\ref{h3n+2}), we have
\begin{flalign*}
	&\left\vert \Sigma _{3n+2}^{3p+1}\right\vert =\left\vert P^{t}\Sigma
_{3n+2}^{3p+1}P\right\vert \\
=&\left\vert
\begin{array}{ccc}
JH_{\overline{n}}^{p} & H_{\overline{n}}^{\overline{p}} & -J^{2}\left( H_{%
\overline{n}}^{\overline{p}}\right) ^{(\overline{n})} \\
H_{\overline{n}}^{\overline{p}} & -J^{2}H_{\overline{n}}^{\overline{p}} &
J\left( H_{\overline{n}}^{\overline{p}}\right) ^{(\overline{n})} \\
-J^{2}\left( H_{\overline{n}}^{\overline{p}}\right) _{(\overline{n})} &
J\left( H_{\overline{n}}^{\overline{p}}\right) _{(\overline{n})} & H_{n}^{%
\overline{p}+1}%
\end{array}%
\right\vert \\
=&\left\vert \left(
\begin{array}{ccc}
JH_{\overline{n}}^{p} & H_{\overline{n}}^{\overline{p}} & -J^{2}\left( H_{%
\overline{n}}^{\overline{p}}\right) ^{(\overline{n})} \\
H_{\overline{n}}^{\overline{p}} & -J^{2}H_{\overline{n}}^{\overline{p}} &
J\left( H_{\overline{n}}^{\overline{p}}\right) ^{(\overline{n})} \\
-J^{2}\left( H_{\overline{n}}^{\overline{p}}\right) _{(\overline{n})} &
J\left( H_{\overline{n}}^{\overline{p}}\right) _{(\overline{n})} & H_{n}^{%
\overline{p}+1}%
\end{array}%
\right) \right. \\
&\qquad \times \left. \left(
\begin{array}{ccc}
I_{\overline{n},\overline{n}} & \boldsymbol{0}_{\overline{n},\overline{n}} &
\boldsymbol{0}_{\overline{n},n} \\
JI_{\overline{n},\overline{n}} & I_{\overline{n},\overline{n}} & J^{2}(I_{%
\overline{n},\overline{n}})^{(\overline{n})} \\
\boldsymbol{0}_{n,\overline{n}} & \boldsymbol{0}_{n,\overline{n}} & I_{n,n}%
\end{array}%
\right) \right\vert \\
=&\left\vert
\begin{array}{ccc}
J(H_{\overline{n}}^{p}+H_{\overline{n}}^{\overline{p}}) & H_{\overline{n}}^{%
\overline{p}} & \boldsymbol{0}_{\overline{n},n} \\
\boldsymbol{0}_{\overline{n},\overline{n}} & -J^{2}H_{\overline{n}}^{%
\overline{p}} & \boldsymbol{0}_{\overline{n},n} \\
\boldsymbol{0}_{n,\overline{n}} & J\left( H_{\overline{n}}^{\overline{p}%
}\right) _{(\overline{n})} & H_{n}^{\overline{p}}+H_{n}^{\overline{p}+1}%
\end{array}%
\right\vert \\
=&(-1)^{n+1}|H_{n+1}^{p+1}|\cdot |\Sigma _{n+1}^{p}|\cdot |\Sigma
_{n}^{p+1}|. &
\end{flalign*}

\smallskip
\noindent
$(L16)$ Combine (\ref{eq:Ks}) and (\ref{h3n}), we have
\begin{flalign*}
&\left\vert \Sigma _{3n}^{3p+2}\right\vert =\left\vert P^{t}\Sigma
_{3n}^{3p+2}P\right\vert \\
=&\left\vert
\begin{matrix}
H_{n}^{\overline{p}} & -J^{2}H_{n}^{\overline{p}} & JH_{n}^{\overline{p}} \\
-J^{2}H_{n}^{\overline{p}} & JH_{n}^{\overline{p}} & H_{n}^{\overline{p}+1}
\\
JH_{n}^{\overline{p}} & H_{n}^{\overline{p}+1} & -J^{2}H_{n}^{\overline{p}+1}%
\end{matrix}%
\right\vert \\
=&\left\vert \left(
\begin{matrix}
H_{n}^{\overline{p}} & -J^{2}H_{n}^{\overline{p}} & JH_{n}^{\overline{p}} \\
-J^{2}H_{n}^{\overline{p}} & JH_{n}^{\overline{p}} & H_{n}^{\overline{p}+1}
\\
JH_{n}^{\overline{p}} & H_{n}^{\overline{p}+1} & -J^{2}H_{n}^{\overline{p}+1}%
\end{matrix}%
\right) \left(
\begin{array}{ccc}
I_{n,n} & \boldsymbol{0}_{n,n} & -JI_{n,n} \\
\boldsymbol{0}_{n,n} & I_{n,n} & \boldsymbol{0}_{n,n} \\
\boldsymbol{0}_{n,n} & JI_{n,n} & I_{n,n}%
\end{array}%
\right) \right\vert \\
=&\left\vert
\begin{matrix}
H_{n}^{\overline{p}} & \boldsymbol{0}_{n,n} & \boldsymbol{0}_{n,n} \\
-J^{2}H_{n}^{\overline{p}} & JH_{n}^{\overline{p}} & H_{n}^{\overline{p}%
}+H_{n}^{\overline{p}+1} \\
JH_{n}^{\overline{p}} & \boldsymbol{0}_{n,n} & -J^{2}(H_{n}^{\overline{p}%
}+H_{n}^{\overline{p}+1})%
\end{matrix}%
\right\vert \\
=&(-1)^{n}|\Sigma _{n}^{p+1}|^{2}\cdot |H_{n}^{p+1}|. &
\end{flalign*}

\smallskip
\noindent
$(L17)$ Combine (\ref{eq:Ks}) and (\ref{h3n+1}), we have
\begin{flalign*}
	&\left\vert \Sigma _{3n+1}^{3p+2}\right\vert =\left\vert P^{t}\Sigma
_{3n+1}^{3p+2}P\right\vert \\
=&\left\vert
\begin{matrix}
H_{\overline{n}}^{\overline{p}} & -J^{2}\left( H_{\overline{n}}^{\overline{p}%
}\right) ^{(\overline{n})} & J\left( H_{\overline{n}}^{\overline{p}}\right)
^{(\overline{n})} \\
-J^{2}\left( H_{\overline{n}}^{\overline{p}}\right) _{(\overline{n})} &
JH_{n}^{\overline{p}} & H_{n}^{\overline{p}+1} \\
J\left( H_{\overline{n}}^{\overline{p}}\right) _{(\overline{n})} & H_{n}^{%
\overline{p}+1} & -J^{2}H_{n}^{\overline{p}+1}%
\end{matrix}%
\right\vert \\
=&\left\vert \left(
\begin{matrix}
H_{\overline{n}}^{\overline{p}} & -J^{2}\left( H_{\overline{n}}^{\overline{p}%
}\right) ^{(\overline{n})} & J\left( H_{\overline{n}}^{\overline{p}}\right)
^{(\overline{n})} \\
-J^{2}\left( H_{\overline{n}}^{\overline{p}}\right) _{(\overline{n})} &
JH_{n}^{\overline{p}} & H_{n}^{\overline{p}+1} \\
J\left( H_{\overline{n}}^{\overline{p}}\right) _{(\overline{n})} & H_{n}^{%
\overline{p}+1} & -J^{2}H_{n}^{\overline{p}+1}%
\end{matrix}%
\right) \right. \\
&\qquad \times \left. \left(
\begin{array}{ccc}
I_{\overline{n},\overline{n}} & \boldsymbol{0}_{\overline{n},n} & -J(I_{%
\overline{n},\overline{n}})^{(\overline{n})} \\
\boldsymbol{0}_{n,\overline{n}} & I_{n,n} & \boldsymbol{0}_{n,n} \\
\boldsymbol{0}_{n,\overline{n}} & JI_{n,n} & I_{n,n}%
\end{array}%
\right) \right\vert \\
=&\left\vert
\begin{matrix}
H_{\overline{n}}^{\overline{p}} & \boldsymbol{0}_{\overline{n},n} &
\boldsymbol{0}_{\overline{n},n} \\
-J^{2}\left( H_{\overline{n}}^{\overline{p}}\right) _{(\overline{n})} &
J(H_{n}^{\overline{p}}+H_{n}^{\overline{p}+1}) & H_{n}^{\overline{p}}+H_{n}^{%
\overline{p}+1} \\
J\left( H_{\overline{n}}^{\overline{p}}\right) _{(\overline{n})} &
\boldsymbol{0}_{n,n} & -J^{2}(H_{n}^{\overline{p}}+H_{n}^{\overline{p}+1})%
\end{matrix}%
\right\vert \\
=&(-1)^{n}|\Sigma _{n}^{p+1}|^{2}\cdot |H_{n+1}^{p+1}|. &
\end{flalign*}

\smallskip
\noindent
$(L18)$ Combine (\ref{eq:Ks}) and (\ref{h3n+2}), we have
\begin{flalign*}
	&\left\vert \Sigma _{3n+2}^{3p+2}\right\vert =\left\vert P^{t}\Sigma
_{3n+2}^{3p+2}P\right\vert \\
=&\left\vert
\begin{matrix}
H_{\overline{n}}^{\overline{p}} & -J^{2}H_{\overline{n}}^{\overline{p}} &
J\left( H_{\overline{n}}^{\overline{p}}\right) ^{(\overline{n})} \\
-J^{2}H_{\overline{n}}^{\overline{p}} & JH_{\overline{n}}^{\overline{p}} &
\left( H_{\overline{n}}^{\overline{p}+1}\right) ^{(\overline{n})} \\
J\left( H_{\overline{n}}^{\overline{p}}\right) _{(\overline{n})} & \left( H_{%
\overline{n}}^{\overline{p}+1}\right) _{(\overline{n})} & -J^{2}H_{n}^{%
\overline{p}+1}%
\end{matrix}%
\right\vert \\
=&\left\vert
\begin{matrix}
H_{\overline{n}}^{\overline{p}} & \boldsymbol{0}_{\overline{n},\overline{n}}
& J\left( H_{\overline{n}}^{\overline{p}}\right) ^{(\overline{n})} \\
-J^{2}H_{\overline{n}}^{\overline{p}} & \boldsymbol{0}_{\overline{n},%
\overline{n}} & \left( H_{\overline{n}}^{\overline{p}+1}\right) ^{(\overline{%
n})} \\
J\left( H_{\overline{n}}^{\overline{p}}\right) _{(\overline{n})} & \left(
\Sigma _{\overline{n}}^{\overline{p}}\right) _{(\overline{n})} &
-J^{2}H_{n}^{\overline{p}+1}%
\end{matrix}%
\right\vert \\
=&0. &
\end{flalign*}


\goodbreak


\end{document}